\pdfoutput=1
\documentclass[12pt]{article}
\usepackage[utf8]{inputenc}
\usepackage[british]{babel}
\usepackage{cmap}
\usepackage{lmodern}

\usepackage{amssymb, amsmath, amsthm}
\usepackage[a4paper,top=25mm,bottom=25mm,left=25mm,right=25mm]{geometry}
\usepackage{ragged2e}

\usepackage{authblk} 
\usepackage{pifont}
\usepackage{graphicx}
\usepackage[dvipsnames,svgnames,table]{xcolor}
\usepackage[figuresright]{rotating}
\usepackage{xtab} 
\usepackage{longtable} 
\usepackage{multirow}
\usepackage{footnote}
\usepackage[stable]{footmisc}
\usepackage{chngpage} 
\usepackage{pdflscape} 
\usepackage[nottoc,notlot,notlof]{tocbibind} 

\usepackage{pgfplots}
\pgfplotsset{every tick label/.append style={font=\footnotesize}}
\pgfplotsset{compat=1.18}
\usepackage{setspace}

\usepackage{array}
\newcolumntype{K}[1]{>{\centering\arraybackslash$}p{#1}<{$}}

\makesavenoteenv{tabular}
\usepackage{tabularx}
\usepackage{booktabs}
\usepackage{threeparttable}
\usepackage[referable]{threeparttablex} 
\newcolumntype{R}{>{\raggedleft\arraybackslash}X}
\newcolumntype{L}{>{\raggedright\arraybackslash}X}
\newcolumntype{C}{>{\centering\arraybackslash}X}
\newcolumntype{A}{>{\columncolor{gray!25}}C}
\newcolumntype{a}{>{\columncolor{gray!25}}c}

\newlength{\tablen}

\usepackage{dcolumn} 
\newcolumntype{.}{D{.}{.}{-1}}

\usepackage{tikz}
\usetikzlibrary{arrows, calc, matrix, patterns, positioning, trees}
\usepackage[semicolon]{natbib}
\usepackage[hyphens]{url}
\usepackage{hyperref} 
\hypersetup{
  colorlinks   = true,    		
  urlcolor     = blue,    		
  linkcolor    = blue,    		
  citecolor    = ForestGreen	
}
\usepackage{microtype}
\usepackage[justification=centerfirst]{caption} 

\usepackage[labelformat=simple]{subcaption}

\DeclareCaptionLabelFormat{parenthesis}{(#2)}
\makeatletter
\renewcommand\p@subfigure{\arabic{figure}.}
\makeatother

\DeclareCaptionLabelFormat{parenthesis}{(#2)}
\makeatletter
\renewcommand\p@subtable{\arabic{table}.}
\makeatother

\usepackage{enumitem}
\usepackage{standalone}

\setlist[itemize]{leftmargin=2.5\parindent}
\setlist[enumerate]{leftmargin=2.5\parindent}

\def\addlegendimage{\csname pgfplots@addlegendimage\endcsname}

\theoremstyle{plain}

\newtheorem{lemma}{Lemma}
\newtheorem{proposition}{Proposition}

\theoremstyle{definition}

\theoremstyle{remark}

\makeatletter
\let\@fnsymbol\@alph
\makeatother

\def\keywords{\vspace{.5em} 
{\noindent \textit{Keywords}: }}

\def\AMS{\vspace{.5em} 
{\noindent \textbf{\emph{MSC} class}: }}

\def\JEL{\vspace{.5em} 
{\noindent \textbf{\emph{JEL} classification number}: }}

\title{A clustering approach for pairwise comparison matrices}
\author{Kolos Csaba \'Agoston\thanks{~Email: \emph{kolos.agoston@uni-corvinus.hu} \newline Corvinus University of Budapest (BCE), Institute of Operations and Decision Sciences, Department of Operations Research and Actuarial Sciences, Budapest, Hungary}
$\qquad \qquad$
S\'andor Boz\'oki\thanks{~Email: \emph{bozoki.sandor@sztaki.hun-ren.hu} \newline
Institute for Computer Science and Control (SZTAKI), Hungarian Research Network (HUN-REN), Laboratory on Engineering and Management Intelligence, Research Group of Operations Research and Decision Systems, Budapest, Hungary \newline
Corvinus University of Budapest (BCE), Institute of Operations and Decision Sciences, Department of Operations Research and Actuarial Sciences, Budapest, Hungary}
$\qquad \qquad$
\href{https://sites.google.com/view/laszlocsato}{L\'aszl\'o Csat\'o}\thanks{~Corresponding author. Email: \emph{laszlo.csato@sztaki.hun-ren.hu} \newline
Institute for Computer Science and Control (SZTAKI), Hungarian Research Network (HUN-REN), Laboratory on Engineering and Management Intelligence, Research Group of Operations Research and Decision Systems, Budapest, Hungary \newline
Corvinus University of Budapest (BCE), Institute of Operations and Decision Sciences, Department of Operations Research and Actuarial Sciences, Budapest, Hungary}} 
\date{\today}

\def\Dedication{
\begin{small}
{\noindent
``\emph{There can be many reasons for searching for representative objects. Not only can these objects provide a characterization of the clusters, but they can often be used for further work or research, especially when it is more economical or convenient to use a small set of $k$ objects instead of the large set one started off with.}''\footnote{~Source: \citet[p.~71]{KaufmanRousseeuw1990}.}
}
\end{small}

\vspace{0.5cm} 
\justify }

\begin{document}

\maketitle
\thispagestyle{empty}
\Dedication

\begin{abstract}
\noindent
We consider clustering in group decision making where the opinions are given by pairwise comparison matrices. In particular, the k-medoids model is suggested to classify the matrices since it has a linear programming problem formulation that may contain any condition on the properties of the cluster centres. Its objective function depends on the measure of dissimilarity between the matrices but not on the weights derived from them. Our methodology provides a convenient tool for decision support, for instance, it can be used to quantify the reliability of the aggregation. The proposed theoretical framework is applied to a large-scale experimental dataset, on which it is able to automatically detect some mistakes made by the decision-makers, as well as to identify a common source of inconsistency.


\AMS{90-10, 90B50, 91B08}

\JEL{C38, C44}
\end{abstract}

\clearpage

\section{Introduction} \label{Sec1}

The fast development of information technology has enabled large-scale group decision making (LSGDM): in many decision making problems, the possible number of decision makers (DMs) now easily reaches thousands \citep{Garcia-ZamoraLabellaDingRodriguezMartinez2022}. However, because the cognitive abilities of humans have not evolved parallel to the exponential growth of data, it is necessary to aggregate the opinions of DMs \citep{TangLiao2021, FanLiangPedryczDong2024}. For this purpose, one of the most widely used techniques is \emph{clustering}: allocating the individual judgements into groups such that judgements in the same group (called a cluster) are more similar to each other than to those in other groups (clusters). 

The Analytic Hierarchy Process (AHP) \citep{Saaty1977, Saaty1980} is a popular multi-criteria decision making (MCDM) methodology, hence, solving group AHP (GAHP) models is almost as old as AHP itself \citep{AczelSaaty1983}. Traditionally, there are two main approaches to aggregating individual preferences and creating a group consensus:
(a) aggregating the individual pairwise comparison matrices (PCMs) and deriving priorities from the common matrix \citep{AczelSaaty1983}; and
(b) deriving priorities from the individual PCMs and aggregating these priorities \citep{BasakSaaty1993}.
\citet{OssadnikSchinkeKaspar2016} recommend the second option as no other aggregation technique satisfies a comparable number of evaluation criteria. However, this approach is sensitive to extreme opinions, the supposed consensus may not accurately portray the true group preference \citep{AmentaIshizakaLucadamoMarcarelliVyas2020}. Furthermore, \citet{DulebaSzadoczki2022} find that distance-based aggregation outperforms conventional methods for a high number of DMs.

Even though cluster analysis of rankings or preferences is sometimes used in the literature, we think the topic deserves further attention. \citet{AlbanoGarcia-LaprestaPlaiaSciandra2024} apply a clustering approach based on preference relations. However, the preferences in this study are always consistent and not expressed via PCMs. A more closely related research field is clustering for fuzzy preference relations and consensus reaching \citep{Garcia-LaprestaPerez-Roman2016, KamisChiclanaLevesley2018, MengTangAn2023, WuXu2018, XuZhongChenZhou2015, ZhangDongHerrera-Viedma2018}.
However, none of these papers use (dis)similarity measures proposed for multiplicative PCMs as the basis of clustering.

On the other hand, we do not know of any studies where individual PCMs are clustered in order to aggregate them. The current paper aims to fill this research gap, which is our main contribution to the extant literature. 

The proposed clustering methodology can be used for several purposes:
\begin{itemize}
\item
it might provide an alternative aggregation procedure if the number of clusters is restricted to one;
\item
it gives information on the ability of aggregated preferences to represent the individual preferences;
\item
it makes it possible to detect some mistakes in the original data, which is far from trivial in LSGDM problems;
\item
it can help to identify common sources of inconsistency.
\end{itemize}
In addition, a crucial advantage of our clustering method is the lack of any requirement on the number of missing entries in the PCMs, which can be especially useful in LSGDM problems where the underlying pairwise comparisons are obtained via a questionnaire that can be finished at every point \citep{FanLiangPedryczDong2024}. Naturally, a complete PCM will have a stronger effect on clustering, but this seems to be a fair compensation for the DMs who have devoted more effort and time to reveal their preferences.

The proposed clustering approach guarantees that the cluster centres are PCMs given by at least one DM. This is also an attractive property because a cluster centre is difficult to interpret if it strongly differs from individual opinions in a MCDM problem.

Last but not least, the suggested $k$-medoids clustering has a linear programming (LP) formulation, which allows for imposing additional restrictions. For example, the inconsistency of the cluster centres can be required to remain below an exogenous threshold.

The theoretical framework is applied to the experimental data of \citet{BozokiDezsoPoeszTemesi2013}.

The paper is structured as follows.
Section~\ref{Sec2} presents the main concepts related to pairwise comparison matrices and examines some dissimilarity measures suggested for them.
Section~\ref{Sec3} introduces cluster analysis and shows the LP formulation of the $k$-medoids clustering.
The numerical results are presented in Section~\ref{Sec4}.
Section~\ref{Sec5} discusses the incorporation of demographic factors that has been suggested by an anonymous reviewer.
Finally, Section~\ref{Sec6} offers concluding remarks.


\section{Pairwise comparison matrices and their similarity} \label{Sec2}

First, the theoretical background of using pairwise comparison matrices and their measures of similarity---the basis of any clustering algorithm---are discussed. Section~\ref{Sec21} introduces pairwise comparison matrices, while Section~\ref{Sec22} overviews some measures to quantify the (dis)similarity of pairwise comparison matrices.

\subsection{Pairwise comparison matrices} \label{Sec21}

Pairwise comparison matrices are often used to determine the cardinal preferences of DMs, who answer questions like ``How many times is a criterion more important than another one?'' or ``How many times is a given alternative better than another one with respect to a fixed criterion?'' These pairwise ratios are collected into the $n \times n$ pairwise comparison matrix ${A} = \left[ a_{ij} \right]$, where $a_{ij} > 0$ and $a_{ij} = 1/a_{ji}$ hold for all $1 \leq i,j \leq n$. 
$A$ is called consistent if $a_{ij} a_{jk} = a_{ik}$ is satisfied for all $1 \leq i,j,k \leq n$. 
Otherwise, when cardinal transitivity is violated, the matrix is called inconsistent.

From a practical point of view, it is natural and reasonable to allow for some degree of inconsistency.
Several inconsistency indices have been proposed to measure the contradictions in a pairwise comparison matrix \citep{BozokiRapcsak2008, Brunelli2018, Csato2019b}. In particular, \citet{Saaty1977} has suggested the so-called \emph{inconsistency index} $\mathit{CI}$ in his celebrated theory of AHP, which is an affine transformation of the dominant eigenvalue $\lambda_{\max}$ of the pairwise comparison matrix $A$:
\[
\mathit{CI}(A) = \frac{\lambda_{\max} - n}{n-1}.
\]
This is divided by the random index $RI_n$, the average $\mathit{CI}$ of a large number of randomly generated $n \times n$ pairwise comparison matrices to obtain the \emph{inconsistency ratio} $\mathit{CR}(A) = \mathit{CI}(A) / RI_n$. According to Saaty, $\mathit{CR}$ should remain below 0.1 to accept the matrix as a reasonable representation of consistent preferences.
 
Analogously, many weighting methods are available to find a positive weight vector $\mathbf{w} = (w_1, w_2, \dots , w_n)$ such that the ratios $w_i/w_j$ give good approximations to the pairwise comparisons $a_{ij}$ \citep{ChooWedley2004}.
The \emph{Logarithmic Least Squares Method} ($\mathit{LLSM}$) \citep{CrawfordWilliams1985, Csato2019a, DeGraan1980, Fichtner1986, Rabinowitz1976} is one of the most popular:
\begin{align*}
\min \sum \limits_{i,j} & \left[\log a_{ij} -\log\left(\frac{w_{i}}{w_{j}}\right)\right]^2 \\
\sum\limits_{i=1}^{n}w_{i} & = 1,   \\
w_{i} & > 0, \qquad i=1,2,\dots ,n.
\end{align*}
The unique solution to this optimisation problem is given by the geometric means of the entries in the rows:
\[
\frac{w_i}{w_j} = \frac{\sqrt[n]{\prod_{k=1}^n a_{ik}}}{\sqrt[n]{\prod_{k=1}^n a_{jk}}}.
\]
Therefore, it is often called the \emph{geometric mean} method.

\subsection{Quantifying the similarity of pairwise comparison matrices} \label{Sec22}

Several ideas exist in the literature for computing the (dis)similarity of two pairwise comparison matrices $A$ and $B$. \citet[p.~389]{CrawfordWilliams1985} essentially introduces the following metric:
\[
D_1(A,B) = \sqrt{\left(\sum_{i=1}^n\sum_{j=1}^n \left( \log(a_{ij})-\log(b_{ij})\right)^2 \right)}.
\]
The original definition takes the sum only for the entries above the diagonal, but this is equivalent to $D_1$ except for a constant factor.
Obviously, $D_1$ is symmetric and equals 0 if and only if $A=B$. It also satisfies the triangle inequality \citep{CrawfordWilliams1985}.

\citet{TekileBrunelliFedrizzi2023} measure the closeness of two complete pairwise comparison matrices by another formula called Manhattan or $L^1$ distance:
\[
D_2(A,B) = \sum_{i=1}^n\sum_{j=1}^n \lvert \log(a_{ij})-\log(b_{ij}) \rvert.
\]
$D_2$ also satisfies the triangle inequality.

\citet{KulakowskiMazurekStrada2022} use another indicator, the so-called compatibility index, which was originally defined by \citet{Saaty2008}:
\[
D_{3u}(A,B) = \frac{1}{n^2} \left( \sum_{i=1}^{n}\sum_{j=1}^n a_{ij}b_{ji} \right).
\]
While $D_{3u}$ is symmetric, $D_{3u}(A,B) \geq n^2$. Hence, it is worth normalising $D_{3u}$ as has been done in \citet{AgostonCsato2024}:
\[
D_3(A,B) = \frac{1}{n^2}\left(\sum_{i=1}^{n}\sum_{j=1}^n \left( a_{ij}b_{ji}-1 \right) \right).
\]
$D_3$ remains symmetric, and $D_3(A,B) = 0$ if and only if $A=B$. However, $D_3$ is not a distance.

\begin{lemma} \label{Lemma1}
Dissimilarity index $D_3$ does not satisfy the triangle inequality.
\end{lemma}

\begin{proof}
It is sufficient to give a counterexample. Consider the following three pairwise comparison matrices:
\[
A =
\left[
\begin{array}{K{1.5em}K{1.5em}K{1.5em}}
1&2&1\\
1/2&1&1\\
1&1&1\\
\end{array}
\right], \qquad
B =
\left[
\begin{array}{K{1.5em}K{1.5em}K{1.5em}}
1&3&1\\
1/3&1&1\\
1&1&1\\
\end{array}
\right], \qquad
C =
\left[
\begin{array}{K{1.5em}K{1.5em}K{1.5em}}
1&4&1\\
1/4&1&1\\
1&1&1\\
\end{array}
\right].
\]
It can be checked that $D_3(A,B) = 1/6$, $D_3(B,C) = 1/12$, but $D_3(A,C) = 1/2 > 1/4 = D_3(A,B) + D_3(B,C)$.
\end{proof}

\citet{KulakowskiMazurekStrada2022} consider three other versions of the compatibility index (all of them are modified to ensure $D(A,A)=0$ and $D(A,B) \geq 0$):
\[
D_4(A,B) = \frac{2}{n(n-1)} \left( \sum_{i=1}^{n-1} \sum_{j=i+1}^n \max \left\{ a_{ij}b_{ji},a_{ji}b_{ij} \right\} -1 \right);
\]
\[
D_5(A,B)= \max \left\{ a_{ij}b_{ji} -1: 1 \leq i,j \leq n \right\};
\]
\[
D_6(A,B) = -\frac{2}{n(n-1)} \left( \sum_{i=1}^{n-1} \sum_{j=i+1}^n \min \left\{ a_{ij}b_{ji},a_{ji}b_{ij} \right\} -1 \right).
\]
Based on $D_4$--$D_6$, another reasonable index could be
\[
D_7(A,B) = -\min \left\{ a_{ij}b_{ji} -1: 1 \leq i,j \leq n \right\}.
\]


\begin{lemma}
Indices $D_4$ and $D_5$ do not satisfy the triangle inequality.
\end{lemma}

\begin{proof}
The PCMs from the proof of Lemma~\ref{Lemma1} can be used:
$D_4(A,B) = 1/6$, $D_4(B,C) = 1/9$ but $D_4(A,C) = 1/3 > 5/18 = D_4(A,B) + D_4(B,C)$, and 
$D_5(A,B) = 1/2$, $D_5(B,C) = 1/3$ but $D_5(A,C) = 1 > 5/6 = D_5(A,B) + D_5(B,C)$.
\end{proof}

\begin{proposition}
Indices $D_6$ and $D_7$ satisfy the triangle inequality, thus, they are distances on the set of pairwise comparison matrices.
\end{proposition}

\begin{proof}
For both $D_6$ and $D_7$, it is sufficient to show that the triangle inequality holds elementwise for any pairwise comparison matrices $A = \left[ a_{ij} \right]$, $B = \left[ b_{ij} \right]$, and $C = \left[ c_{ij} \right]$. $1 = a_{ij} \leq b_{ij} \leq c_{ij}$ can be assumed without loss of generality. Then, by elementary calculus, the following three inequalities hold:
\[
1- \frac{a_{ij}}{b_{ij}} \quad + \quad 1 - \frac{b_{ij}}{c_{ij}}  \quad \geq \quad 1 - \frac{a_{ij}}{c_{ij}};
\]
\[
1- \frac{a_{ij}}{c_{ij}} \quad + \quad 1 - \frac{b_{ij}}{c_{ij}} \quad \geq \quad 1 - \frac{a_{ij}}{b_{ij}};
\]
\[
1- \frac{a_{ij}}{b_{ij}} \quad + \quad 1 - \frac{a_{ij}}{c_{ij}} \quad \geq \quad 1 - \frac{b_{ij}}{c_{ij}}.
\]
\end{proof}

\citet{Fichtner1984} verifies that the popular eigenvector method \citep{Saaty1977} can be defined by a metric, too. However, although it satisfies the requirements of a distance, it is discontinuous.


All of the above dissimilarity measures $D_1$--$D_7$ are invariant under the relabeling of the alternatives. Consequently, they are invariant to transposition, that is,
\[
D_i(A,B) = D_i(A^\top,B^\top),
\]
where $A^\top$ is the transpose of the pairwise comparison matrix $A$.

Since the input of the $k$-medoids problem is the dissimilarity matrix of the objects, any of the indices $D_1$--$D_7$ can be used in the proposed methodology. Our numerical experiment will consider $D_1$ and $D_3$.

\section{Cluster analysis and the $k$-medoids problem} \label{Sec3}

Clustering, the problem of dividing a set of objects into groups (clusters) such that any object is more similar to an arbitrary object in its group than to any object in a distinct group, is one of the most popular methods in exploratory data science. Although no ``perfect'' clustering algorithm exists, perhaps the most natural solution is $k$-means clustering, when each object belongs to the cluster with the nearest mean.

However, this leads to an NP-hard non-convex discrete optimisation problem.
Therefore, most state-of-the-art statistical and data science software implement a heuristic iterative algorithm that alternates two steps: i) assigning all points to the nearest cluster centre; and ii) recalculating the cluster centres.
The operation in step ii) is quite trivial in Euclidean spaces (take the average in every dimension), but it may be challenging if the objects do not belong to a multidimensional Euclidean space \citep{MajstorovicSaboJungKlaric2018}.
Furthermore, the heuristic iterative algorithm above can converge to a local rather than global optimum. Hence, \citet{AgostonEisenberg-Nagy2024} provide a mixed integer linear programming (MILP) formulation to obtain an exact solution for the minimum sum-of-clustering problem that allows for adding arbitrary linear constraints and can be solved up to 150 objects.

An alternative approach to $k$-means clustering is the $k$-medoids problem \citep{KaufmanRousseeuw1986, SchubertRousseeuw2019}. Here, all cluster centres should be objects themselves, allowing for easier interpretability of the cluster centres. This seems to be an important advantage in LSGDM since the DMs may not be willing to accept a solution that has not been suggested by any of them. In addition, $k$-medoids can be used with an arbitrary dissimilarity measure, while $k$-means generally requires Euclidean distance. Last but not least, minimising the sum of pairwise dissimilarities instead of the sum of squared Euclidean distances makes the $k$-medoids problem less sensitive to outliers than the $k$-means problem.

For unknown reasons, the $k$-medoids model is not widely used and is not implemented in usual statistical software packages. Nonetheless, the problem can be formulated as an LP model, which is described in the following.


Assume that the $m \times m$ matrix $\Delta = \left[ \delta_{ij} \right]$ contains the (pairwise) dissimilarities between any two of the $m$ objects (individual PCMs). For example, $\delta_{ij} = D_1 \left( A^{(i)}, A^{(j)} \right)$ if measure $D_1$ is used to determine the dissimilarity of the preferences given by matrices $A^{(i)}$ and $A^{(j)}$ of decision makers $i$ and $j$. Let $M$ denote the set of integers $\{1, \dots, m \}$.

The $k$-medoids problem is as follows (see, for instance, \citet{Vinod1969}):
\begin{align}
\sum_{i=1}^m \sum_{j=1}^m \delta_{ij}x_{ij} & \rightarrow \min &  \label{eq_MSSof}\\
\mbox{s.t.} \qquad \sum_{j=1}^m x_{ij} & = 1& \forall\; i \in M \label{eq_KMedC1}\\
x_{ij}&\le y_j& \forall\; i,j \in M \label{eq_KMedC2}\\
\sum_{j=1}^m y_j & = k &  \label{eq_KMedC3}\\
x_{ij}& \ge  0& \forall\; i,j \in M \label{eq_KMedC4}\\
y_{j}&\in\{0,1\}& \forall\; j \in M \label{eq_KMedC5}
\end{align}
In this formulation, the binary variable $y_j$ equals one if object $j$ is a cluster centre (and zero otherwise), while $x_{ij}$ equals one if object $i$ is placed in the cluster whose centre is object $j$ (and zero otherwise). Constraint \eqref{eq_KMedC3} ensures that there are $k$ different cluster centres. According to \eqref{eq_KMedC1}, each object is classified into exactly one cluster. Finally, due to the constraints \eqref{eq_KMedC2}, object $i$ can be associated with object $j$ ($x_{ij}=1$) only if object $j$ is a cluster centre, namely, $y_j=1$.

The LP model \eqref{eq_MSSof}--\eqref{eq_KMedC5} can be solved with standard solvers in small and medium size instances.
The above formulation allows any PCM to be a cluster centre. However, further restrictions can be easily added to the LP model, for example, by requiring that the inconsistency ratio $\mathit{CR}$ of each cluster centre is below an exogenous threshold.

Originally, clustering was intended to form groups in Euclidean spaces, but the underlying idea can be applied in other fields. Clustering algorithms appear in network science, where the goal is to classify nodes in a graph such that the nodes in a given cluster are connected more strongly to each other than to nodes in other clusters. Other objects can also be clustered, for example, curves or mortality tables. 

To conclude, cluster analysis requires a clustering algorithm and a dissimilarity measure, which is not necessarily a distance. This paper considers the $k$-medoids problem due to its advantages discussed above.

\section{Numerical results} \label{Sec4}

\begin{table}[t!]
\centering
\caption{Sample sizes in the experimental database}
\label{Table1}
    \rowcolors{1}{}{gray!20}
\begin{tabularx}{\textwidth}{CcCc} \toprule
Label & Type & Size ($n$) & Number of PCMs ($m$) \\ \bottomrule
M4 & Objective (map) & 4 &  66\\
M6 & Objective (map) & 6 &  77\\
M8 & Objective (map) & 8 &  82\\
S4 & Subjective (summer house) & 4 &  68\\
S6 & Subjective (summer house) & 6 &  77\\
S8 & Subjective (summer house) & 8 &  78\\ \bottomrule
\end{tabularx}
\end{table}

\citet{BozokiDezsoPoeszTemesi2013} have collected hundreds of pairwise comparison matrices in a controlled experiment where even the questioning order is known for each matrix. However, here we use only the final complete pairwise comparison matrices. The matrices are distinguished by their size ($n=4$, $n=6$, $n=8$ alternatives) and type (objective, where the students compared maps; subjective, where the students compared summer houses). The sample sizes are reported in Table~\ref{Table1}.

\subsection{Analysing a sample of objective type} \label{Sec41}

In the case of dataset M4, the ``correct'' consistent pairwise comparison matrix provided by the ratios of country areas that appear on the map are known:
\[
\left[
\begin{array}{cccc}
1 &	1.691&	0.282&	0.770\\
0.591&	1 &	0.167&	0.455\\
3.544&	5.991&	1 &	2.725\\
1.300&	2.198&	0.367&	1
\end{array}
\right].
\]
For $k=2$ clusters, the cluster centres coincide for both dissimilarity measures $D_1$ and $D_3$:
\[
CC_1^{(1)} = CC_3^{(1)} =
\left[
\begin{array}{cccc}
1 &	1.500&	0.286&	0.833\\
0.667&	1 &	0.154&	0.435\\
3.500&	6.500&	1 &	2.300\\
1.200&	2.300&	0.435&	1 
\end{array}
\right],
\]
and
\[
CC_1^{(2)} = CC_3^{(2)} =
\left[
\begin{array}{cccc}
1 &	1.200&	0.400&	0.909\\
0.833&	1 &	0.200&	0.500\\
2.500&	5.000&	1 &	2.300\\
1.100&	2.000&	0.500&	1 
\end{array}
\right].
\]
The two clusters are the same: 47 matrices belong to first and 19 to the second group.

\begin{figure}[t!]
\centering

\begin{tikzpicture}

\draw[very thick](-0.5,0) -- (5.5,0);

\draw[->, very thick](0,-0.5) -- (0,5.5);

\draw[very thick](-0.1,1.43) -- (0.1,1.43);
\node at (-0.5,1.43) {\footnotesize{0.01}};
\draw[very thick](-0.1,2.86) -- (0.1,2.86);
\node at (-0.5,2.86) {\footnotesize{0.02}};
\draw[very thick](-0.1,4.29) -- (0.1,4.29);
\node at (-0.5,4.29) {\footnotesize{0.03}};

\filldraw[fill=blue!25]
(1,0.34) -- (1,1.52) -- (2,1.52) --  (2,0.34) -- (1,0.34);

\draw[very thick](1,0.69) -- (2,0.69);

\draw[dashed] (1.5,0.05)--(1.5,0.34);
\draw (1.2,0.05)--(1.8,0.05);

\draw[dashed] (1.5,1.52)--(1.5,3.06);
\draw (1.2,3.06)--(1.8,3.06);

\draw (1.45,4.80)--(1.55,4.9);
\draw (1.55,4.80)--(1.45,4.9);

\node at (1.5,-0.5) {\footnotesize{Cluster \#1}};

\filldraw[fill=blue!25]
(4,0.48) -- (4,1.24) -- (5,1.24) --  (5,0.48) -- (4,0.48);

\draw[very thick](4,1.08) -- (5,1.08);

\draw[dashed] (4.5,0.05)--(4.5,0.48);
\draw (4.2,0.05)--(4.8,0.05);

\draw[dashed] (4.5,1.21)--(4.5,1.71);
\draw (4.2,1.71)--(4.8,1.71);

\draw (4.45,2.59)--(4.55,2.69);
\draw (4.55,2.59)--(4.45,2.69);

\node at (4.5,-0.5) {\footnotesize{Cluster \#2}};
\end{tikzpicture}

\caption{Distribution of inconsistency ratios $\mathit{CR}$, dataset M4, $k=2$ clusters}
\label{Fig1}
\end{figure}
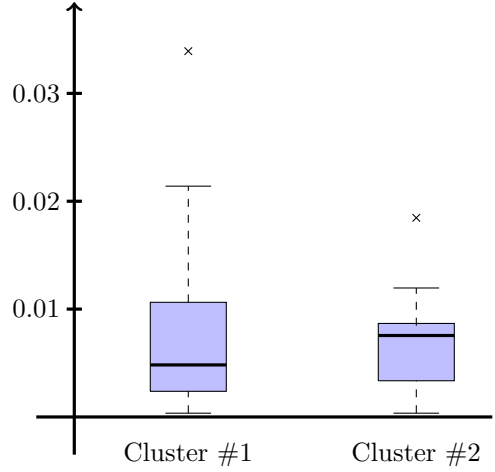


Figure~\ref{Fig1} shows boxplots for the inconsistency ratios $\mathit{CR}$ in the two clusters.

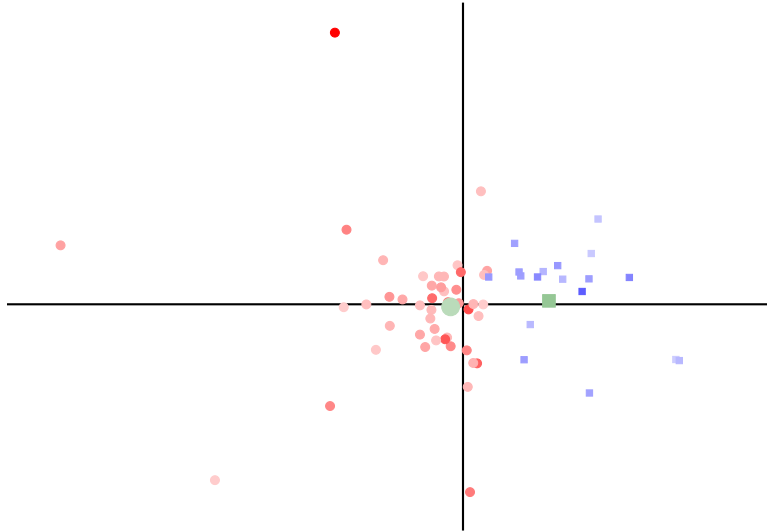
\begin{figure}
\centering

\begin{tikzpicture}[scale=2]
\draw[thick] (-3,0)--(2,0);
\draw[thick] (0,-1.5)--(0,2);

\node[draw,circle,scale=0.3,color=red!27.743173036104,fill=red!27.743173036104] at (-0.104582758960714,-0.220236914354549) {};
\node[draw,circle,scale=0.3,color=red!20,fill=red!20] at (-1.63196231362193,-1.16585908230453) {};
\node[draw,circle,scale=0.3,color=red!25.0009005547447,fill=red!25.0009005547447] at (-0.635920195171941,-0.000846673345919702) {};
\node[draw,circle,scale=0.3,color=red!28.1701883564313,fill=red!28.1701883564313] at (0.0312937108696674,-0.547533312478262) {};
\node[draw,circle,scale=0.3,color=red!21.5400465734548,fill=red!21.5400465734548] at (-0.573045426326024,-0.301791543521814) {};
\node[draw,circle,scale=0.3,color=red!70.3742776343295,fill=red!70.3742776343295] at (0.0362815107874781,-0.0339235070441709) {};
\node[draw,circle,scale=0.3,color=red!22.5681911326754,fill=red!22.5681911326754] at (-0.0359057205065555,0.258771589619418) {};
\node[draw,circle,scale=0.3,color=red!32.8922348708001,fill=red!32.8922348708001] at (-0.186703935219404,-0.163481670083828) {};
\node[draw,circle,scale=0.3,color=red!48.9415457975426,fill=red!48.9415457975426] at (-0.766386756936899,0.494497402862245) {};
\node[draw,circle,scale=0.3,color=red!52.2050937575925,fill=red!52.2050937575925] at (-0.0987671515162197,0.0139819093486296) {};
\node[draw,circle,scale=0.3,color=red!35.1304297549793,fill=red!35.1304297549793] at (-0.24875129647102,-0.283226237909117) {};
\node[draw,circle,scale=0.3,color=red!52.2165963874199,fill=red!52.2165963874199] at (0.0461811761688024,-1.24516275853874) {};
\node[draw,circle,scale=0.3,color=red!45.6981272170028,fill=red!45.6981272170028] at (-0.0811214955217393,-0.278184425835227) {};
\node[draw,circle,scale=0.3,color=red!100,fill=red!100] at (-0.842834943959493,1.80090563228448) {};
\node[draw,circle,scale=0.3,color=red!28.7123616616373,fill=red!28.7123616616373] at (0.0683191667579885,0.00152595558172792) {};
\node[draw,circle,scale=0.3,color=red!29.084844041208,fill=red!29.084844041208] at (-0.481511075517613,-0.142773297516115) {};
\node[draw,circle,scale=0.3,color=red!25.384197169623,fill=red!25.384197169623] at (-0.283138937177438,-0.00709627691763844) {};
\node[draw,circle,scale=0.3,color=red!28.3167787968792,fill=red!28.3167787968792] at (-0.158654809676147,0.184849851914826) {};
\node[draw,circle,scale=0.3,color=red!36.3636666740629,fill=red!36.3636666740629] at (-2.64758082036684,0.390824590478015) {};
\node[draw,circle,scale=0.3,color=red!20.0560485123925,fill=red!20.0560485123925] at (-0.784569960941716,-0.0193533156809578) {};
\node[draw,circle,scale=0.3,color=red!66.3916130421572,fill=red!66.3916130421572] at (-0.117473648919792,-0.231955678597323) {};
\node[draw,circle,scale=0.3,color=red!23.0840409684335,fill=red!23.0840409684335] at (-0.176919800500813,-0.239383190357657) {};
\node[draw,circle,scale=0.3,color=red!28.3981807755842,fill=red!28.3981807755842] at (-0.125113279431071,0.184202640743869) {};
\node[draw,circle,scale=0.3,color=red!41.5609204288983,fill=red!41.5609204288983] at (-0.484056769522291,0.0495359315574239) {};
\node[draw,circle,scale=0.3,color=red!57.7613364551527,fill=red!57.7613364551527] at (-0.202958203956108,0.0403846304982148) {};
\node[draw,circle,scale=0.3,color=red!35.7407428926191,fill=red!35.7407428926191] at (0.158534576877362,0.221999100903998) {};
\node[draw,circle,scale=0.3,color=red!25.6829218685413,fill=red!25.6829218685413] at (0.102261163790481,-0.077997259751059) {};
\node[draw,circle,scale=0.3,color=red!40.2340414596807,fill=red!40.2340414596807] at (-0.0279056975006285,0.00724571671139551) {};
\node[draw,circle,scale=0.3,color=red!28.9075087053951,fill=red!28.9075087053951] at (-0.525466270694212,0.292247330963276) {};
\node[draw,circle,scale=0.3,color=red!27.310853054461,fill=red!27.310853054461] at (-0.207717390256928,-0.0371336643339072) {};
\node[draw,circle,scale=0.3,color=red!25.184311053893,fill=red!25.184311053893] at (-0.124554573508315,0.085929567446761) {};
\node[draw,circle,scale=0.3,color=red!58.9436438749534,fill=red!58.9436438749534] at (-0.0141600209481007,0.213111913840607) {};
\node[draw,circle,scale=0.3,color=red!28.1858387831695,fill=red!28.1858387831695] at (-0.21494973536453,-0.0944140424359473) {};
\node[draw,circle,scale=0.3,color=red!34.4358547042338,fill=red!34.4358547042338] at (-0.205300290926322,0.123690475352242) {};
\node[draw,circle,scale=0.3,color=red!47.963375906632,fill=red!47.963375906632] at (0.0247129064717145,-0.305324187173667) {};
\node[draw,circle,scale=0.3,color=red!62.5105132098779,fill=red!62.5105132098779] at (0.0927389014234722,-0.391502571242273) {};
\node[draw,circle,scale=0.3,color=red!25.0212007136945,fill=red!25.0212007136945] at (0.138527537948931,0.194740716674971) {};
\node[draw,circle,scale=0.3,color=red!34.5119703222414,fill=red!34.5119703222414] at (-0.28318796801274,-0.200902533282209) {};
\node[draw,circle,scale=0.3,color=red!37.9033584920585,fill=red!37.9033584920585] at (-0.14427232319798,0.112599497906529) {};
\node[draw,circle,scale=0.3,color=red!30.4745515233588,fill=red!30.4745515233588] at (0.0673277482947958,-0.389013959013151) {};
\node[draw,circle,scale=0.3,color=red!21.5400465538363,fill=red!21.5400465538363] at (-0.262545744718808,0.18604758609582) {};
\node[draw,circle,scale=0.3,color=red!44.3276766407232,fill=red!44.3276766407232] at (-0.0438327816752776,0.096394750717656) {};
\node[draw,circle,scale=0.3,color=red!47.2475155101648,fill=red!47.2475155101648] at (-0.874610455668206,-0.674485396559464) {};
\node[draw,circle,scale=0.3,color=red!34.1078735833547,fill=red!34.1078735833547] at (-0.398052651819794,0.0321804710225935) {};
\node[draw,circle,scale=0.3,color=red!25.1144268582074,fill=red!25.1144268582074] at (0.117922609981497,0.749029549008262) {};
\node[draw,circle,scale=0.3,color=red!20.0000000000016,fill=red!20.0000000000016] at (0.133263361047561,-0.00107656045657967) {};
\node[draw,circle,scale=0.6,color=ForestGreen!31.362769464017,fill=ForestGreen!31.362769464017] at (-0.0814399220056556,-0.0172128861565627) {};

\node[draw,rectangle,scale=0.3,color=blue!38.6142749867989,fill=blue!38.6142749867989] at (0.401811111834794,-0.367520128269894) {};
\node[draw,rectangle,scale=0.3,color=blue!29.1360792166344,fill=blue!29.1360792166344] at (0.527609569680219,0.217049069913116) {};
\node[draw,rectangle,scale=0.3,color=blue!38.0508418730004,fill=blue!38.0508418730004] at (0.828758691004628,0.168639581688849) {};
\node[draw,rectangle,scale=0.3,color=blue!63.5047571612574,fill=blue!63.5047571612574] at (0.78359165831749,0.0848983889608239) {};
\node[draw,rectangle,scale=0.3,color=blue!21.2752686389075,fill=blue!21.2752686389075] at (0.84366977989118,0.336633454095644) {};
\node[draw,rectangle,scale=0.3,color=blue!28.1155938082711,fill=blue!28.1155938082711] at (0.655931506672911,0.166556627166551) {};
\node[draw,rectangle,scale=0.3,color=blue!48.1720601185208,fill=blue!48.1720601185208] at (1.09445109077528,0.177944524490718) {};
\node[draw,rectangle,scale=0.3,color=blue!37.294818793277,fill=blue!37.294818793277] at (0.339513038008784,0.403859117803098) {};
\node[draw,rectangle,scale=0.3,color=blue!37.7216955458462,fill=blue!37.7216955458462] at (0.380125193308673,0.188039117983584) {};
\node[draw,rectangle,scale=0.3,color=blue!38.4239724001293,fill=blue!38.4239724001293] at (0.368751956475115,0.214281499458956) {};
\node[draw,rectangle,scale=0.3,color=blue!19.9999999999996,fill=blue!19.9999999999996] at (1.40047846133009,-0.366051373377257) {};
\node[draw,rectangle,scale=0.3,color=blue!36.6359543162581,fill=blue!36.6359543162581] at (0.831844836429085,-0.587996749811936) {};
\node[draw,rectangle,scale=0.3,color=blue!40.2018288979405,fill=blue!40.2018288979405] at (0.622901961769595,0.256444610718295) {};
\node[draw,rectangle,scale=0.3,color=blue!40.5431211961626,fill=blue!40.5431211961626] at (0.169184477269212,0.180827832648823) {};
\node[draw,rectangle,scale=0.3,color=blue!46.1763414616583,fill=blue!46.1763414616583] at (0.490387484342199,0.181158241218207) {};
\node[draw,rectangle,scale=0.3,color=blue!27.304505015092,fill=blue!27.304505015092] at (0.442457416088837,-0.134340510091306) {};
\node[draw,rectangle,scale=0.3,color=blue!27.7463948804279,fill=blue!27.7463948804279] at (1.4230440539301,-0.373135682036155) {};
\node[draw,rectangle,scale=0.3,color=blue!23.3021651048986,fill=blue!23.3021651048986] at (0.88854624127079,0.565302757179672) {};
\node[draw,rectangle,scale=0.6,color=ForestGreen!47.8456038325649,fill=ForestGreen!47.8456038325649] at (0.565532227700522,0.0225837536179185) {};
\end{tikzpicture}

\caption{Approximated coordinates of PCMs, dataset M4, measure $D_3$ \\ \vspace{0.25cm}
\footnotesize{\emph{Notes}: The two clusters are represented by red dots and blue squares. A darker mark is associated with a higher inconsistency. The cluster centers are green and have double sizes.}}
\label{Fig2}
\end{figure}


We have used multidimensional scaling (MDS) to visualise the pairwise comparison matrices \citep{Kruskal1964, KruskalWish1978}. This technique places each matrix into a lower-dimensional space such that the original distances are preserved to the extent possible.
In Figure~\ref{Fig2}, red dots and blue squares represent the two clusters. A darker mark is associated with a higher inconsistency. The two cluster centres are green and have double sizes. Clearly, the clusters are not determined by the inconsistency of the matrices.
The axes are artificial measures in a two-dimensional Euclidean space without a clear interpretation, as provided by the \texttt{cmdscale()} function of \texttt{R}.
 
The main findings can be summarised as follows:
\begin{itemize}
\item
The clusters are quite balanced with respect to the number of matrices in them;
\item
The results are insensitive to the dissimilarity measure ($D_1$ or $D_3$) used;
\item
The two groups are not distinguished by the level of inconsistency, although the variance of $\mathit{CR}$ is somewhat higher for the first group.
\end{itemize}

The analysis has been repeated with higher numbers of clusters. A strong relationship remains between the resulting groups for different values of $k$, as well as for the two dissimilarity measures $D_1$ and $D_3$.

In sample M8 with measure $D_3$ and $k=2$ clusters, an interesting example has been found that uncovers a possible application of our methodology.
The cluster centres are:
\[
CC_3^{(1)} =
\left[
\begin{array}{cccccccc}
1 & 1.800& 0.769& 0.556& 1.800&  5.000& 1.100& 3.000 \\
0.556& 1 & 0.455& 0.313& 1.100&  2.700& 0.909& 1.500 \\
1.300& 2.200& 1 & 0.833& 3.000&  7.000& 1.800& 3.300 \\
1.800& 3.200& 1.200& 1 & 5.000& \textbf{10.100} & 2.200& 6.000 \\
0.556& 0.909& 0.333& 0.200& 1 &  2.300& 0.625& 1.400 \\
0.200& 0.370& 0.143& 0.099& 0.435&  1 & 0.250& 0.769 \\
0.909& 1.100& 0.556& 0.455& 1.600&  4.000& 1 & 2.200 \\
0.333& 0.667& 0.303& 0.167& 0.714&  1.300& 0.455& 1 \\
\end{array}
\right],
\]
and
\[
CC_3^{(2)} =
\left[
\begin{array}{cccccccc}
1 & 1.500& 0.500&  0.500& 3.000& 6.000& 1.200& 3.000 \\
0.667& 1 & 0.500&  0.333& 1.500& 3.000& 0.667& 1.500 \\
2.000& 2.000& 1 &  0.667& 5.000& 8.000& 1.500& 4.000 \\
2.000& 3.000& 1.500&  1 & 4.500& \textbf{0.100} & 2.000& 5.000 \\
0.333& 0.667& 0.200&  0.222& 1 & 2.000& 0.667& 1.200 \\
0.167& 0.333& 0.125& 10.000& 0.500& 1 & 0.286& 0.500 \\
0.833& 1.500& 0.667&  0.500& 1.500& 3.500& 1 & 2.500 \\
0.333& 0.667& 0.250&  0.200& 0.833& 2.000& 0.400& 1  \\
\end{array}
\right].
\]
The first cluster contains 81 matrices and the second cluster consists of only one. The two cluster centres are similar except for the preference between the fourth and sixth alternatives (countries), highlighted in bold font. Since the pairwise comparisons show the ratio of the area of two countries, the only PCM in the second cluster is almost certain to contain a typo: the student has thought that country 4 is ten times larger than country 6 but mistakenly written the reciprocal 1/10, which is a standard mistake. Unsurprisingly, the associated PCM fundamentally differs from all other matrices, and the presented clustering approach is able to detect the outlier without any other specification. This can be highly advantageous in LSGDM problems.

\subsection{Analysing a sample of subjective type} \label{Sec42} 

When the students have compared summer houses, no ``natural'' PCM exists around which the opinions are centred.

Since sample S4 contains four different houses, we have started the analysis with four clusters, supposing that each alternative will be the best in one cluster.

For the measure $D_1$, the cluster centres are as follows:
\[
\mathit{CC}_1^{(1)} = \left[
\begin{array}{cccc}
1 &	2.000&	7.000&	3.000\\
0.500&	1 &	5.000&	2.000\\
0.143&	0.200&	1 &	0.500\\
0.333&	0.500&	2.000&	1 
\end{array}
\right], \qquad
\mathit{CC}_1^{(2)} = \left[
\begin{array}{cccc}
1 &	1.500&	0.500&	0.333\\
0.667&	1 &	0.333&	0.333\\
2.000&	3.000&	1 &	0.667\\
3.000&	3.000&	1.500&	1 
\end{array}
\right],
\]
\[
\mathit{CC}_1^{(3)} = \left[
\begin{array}{cccc}
1 &	3.000&	5.000&	0.400\\
0.333&	1 &	3.000&	0.143\\
0.200&	0.333&	1 &	0.111\\
2.500&	7.000&	9.000&	1 
\end{array}
\right], \qquad
\mathit{CC}_1^{(4)} = \left[
\begin{array}{cccc}
1 &	6.000&	3.000&	3.000\\
0.167&	1 &	0.200&	0.200\\
0.333&	5.000&	1 &	0.500\\
0.333&	5.000&	2.000&	1 
\end{array}
\right].
\]
On the other hand, for the measure $D_3$:
\[
\mathit{CC}_3^{(1)} = \left[
\begin{array}{cccc}
1 &	2.000&	7.000&	3.000\\
0.500&	1 &	5.000&	2.000\\
0.143&	0.200&	1 &	0.500\\
0.333&	0.500&	2.000&	1 
\end{array}
\right], \qquad
\mathit{CC}_3^{(2)} = \left[
\begin{array}{cccc}
1 &	1.500&	0.333&	0.200\\
0.667&	1 &	0.333&	0.200\\
3.000&	3.000&	1 &	0.500\\
5.000&	5.000&	2.000&	1 
\end{array}
\right].
\]
\[
\mathit{CC}_3^{(3)} = \left[
\begin{array}{cccc}
1 &	4.000&	1.500&	2.000\\
0.250&	1 &	0.333&	0.667\\
0.667&	3.000&	1 &	3.000\\
0.500&	1.500&	0.333&	1 
\end{array}
\right], \qquad
\mathit{CC}_3^{(4)} = \left[
\begin{array}{cccc}
1 &	5.000&	3.000&	1.000\\
0.200&	1 &	0.500&	0.143\\
0.333&	2.000&	1 &	0.200\\
1.000&	7.000&	5.000&	1 
\end{array}
\right].
\]
Note that the centres of the first clusters (matrices $\mathit{CC}_1^{(1)}$ and $\mathit{CC}_3^{(1)}$) coincide. The cluster sizes are 26, 15, 9, 18 for index $D_1$ and 27, 9, 12, 20 for index $D_3$.

\begin{table}[t!]
\centering
\caption{Contingency table for sample S4}
\label{Table2}
\begin{threeparttable}
    \rowcolors{1}{}{gray!20}
\begin{tabularx}{0.6\textwidth}{CCCCCC} \toprule
 & $\mathit{CL}_3^{(1)}$ & $\mathit{CL}_3^{(2)}$ & $\mathit{CL}_3^{(3)}$ & $\mathit{CL}_3^{(4)}$ & $\sum$ \\ \bottomrule
$\mathit{CL}_1^{(1)}$ &   25    & 0     & 1     & 0 & 26 \\
$\mathit{CL}_1^{(2)}$ &    0     & 9     & 5     & 1 & 15 \\
$\mathit{CL}_1^{(3)}$ &    2     & 0     & 0     & 7 & 9 \\
$\mathit{CL}_1^{(4)}$ &    0     & 0     & 6     & 12 & 18 \\ \bottomrule
$\sum$ & 27 & 9 & 12 & 20 & 68 \\ \toprule
\end{tabularx}
\begin{tablenotes} \footnotesize
\item
Four clusters, measures $D_1$ (row) and $D_3$ (column)
\end{tablenotes}
\end{threeparttable}
\end{table}

The contingency table of the two groupings is shown in Table~\ref{Table2}. The similarity of the clusters is lower than in the objective task (Section~\ref{Sec41}), but the largest cluster almost coincides.

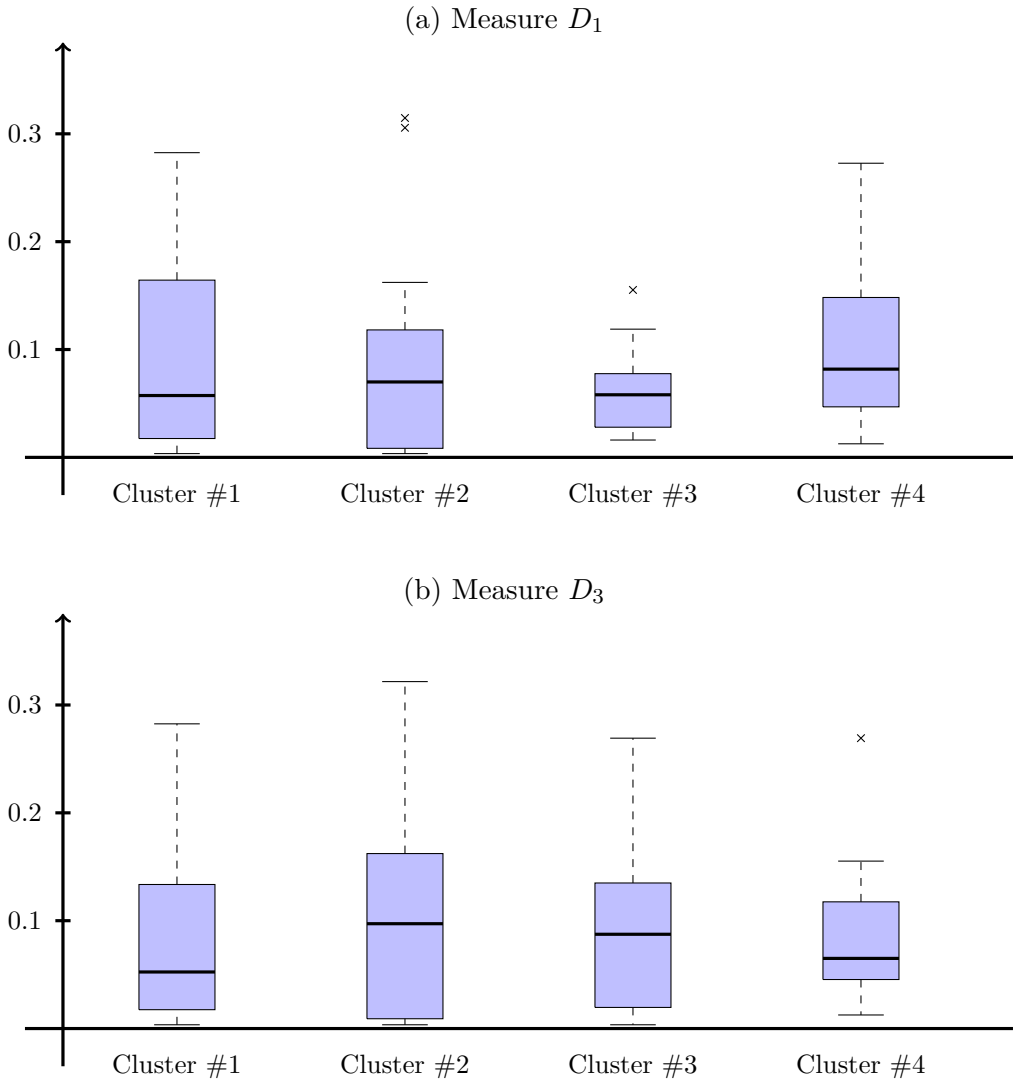
\begin{figure}[t!]
\centering

\begin{subfigure}{\textwidth}
\centering
\caption{Measure $D_1$}
\label{Fig3a}  

\begin{tikzpicture}

\draw[very thick](-0.5,0) -- (12.5,0);

\draw[->, very thick](0,-0.5) -- (0,5.5);

\draw[very thick](-0.1,1.43) -- (0.1,1.43);
\node at (-0.5,1.43) {\footnotesize{0.1}};
\draw[very thick](-0.1,2.86) -- (0.1,2.86);
\node at (-0.5,2.86) {\footnotesize{0.2}};
\draw[very thick](-0.1,4.29) -- (0.1,4.29);
\node at (-0.5,4.29) {\footnotesize{0.3}};

\filldraw[fill=blue!25]
(1,0.25) -- (1,2.35) -- (2,2.35) --  (2,0.25) -- (1,0.25);

\draw[very thick](1,0.82) -- (2,0.82);

\draw[dashed] (1.5,0.05)--(1.5,0.25);
\draw (1.2,0.05)--(1.8,0.05);

\draw[dashed] (1.5,2.35)--(1.5,4.04);
\draw (1.2,4.04)--(1.8,4.04);

\node at (1.5,-0.5) {\footnotesize{Cluster \#1}};

\filldraw[fill=blue!25]
(4,0.12) -- (4,1.69) -- (5,1.69) --  (5,0.12) -- (4,0.12);

\draw[very thick](4,1.00) -- (5,1.00);

\draw[dashed] (4.5,0.05)--(4.5,0.12);
\draw (4.2,0.05)--(4.8,0.05);

\draw[dashed] (4.5,1.69)--(4.5,2.32);
\draw (4.2,2.32)--(4.8,2.32);

\draw (4.45,4.45)--(4.55,4.55);
\draw (4.55,4.45)--(4.45,4.55);

\draw (4.45,4.32)--(4.55,4.42);
\draw (4.55,4.32)--(4.45,4.42);

\node at (4.5,-0.5) {\footnotesize{Cluster \#2}};

\filldraw[fill=blue!25]
(7,0.4) -- (7,1.11) -- (8,1.11) --  (8,0.4) -- (7,0.4);

\draw[very thick](7,0.83) -- (8,0.83);

\draw[dashed] (7.5,0.23)--(7.5,0.4);
\draw (7.2,0.23)--(7.8,0.23);

\draw[dashed] (7.5,1.11)--(7.5,1.70);
\draw (7.2,1.70)--(7.8,1.70);

\node at (7.5,-0.5) {\footnotesize{Cluster \#3}};

\draw (7.45,2.17)--(7.55,2.27);
\draw (7.55,2.17)--(7.45,2.27);

\filldraw[fill=blue!25]
(10,0.67) -- (10,2.12) -- (11,2.12) --  (11,0.67) -- (10,0.67);

\draw[very thick](10,1.17) -- (11,1.17);

\draw[dashed] (10.5,0.18)--(10.5,0.67);
\draw (10.2,0.18)--(10.8,0.18);

\draw[dashed] (10.5,2.12)--(10.5,3.90);
\draw (10.2,3.90)--(10.8,3.90);

\node at (10.5,-0.5) {\footnotesize{Cluster \#4}};

\end{tikzpicture}
\end{subfigure}

\vspace{0.5cm}
\begin{subfigure}{\textwidth}
\centering
\caption{Measure $D_3$}
\label{Fig3b}  

\begin{tikzpicture}

\draw[very thick](-0.5,0) -- (12.5,0);

\draw[->, very thick](0,-0.5) -- (0,5.5);

\draw[very thick](-0.1,1.43) -- (0.1,1.43);
\node at (-0.5,1.43) {\footnotesize{0.1}};
\draw[very thick](-0.1,2.86) -- (0.1,2.86);
\node at (-0.5,2.86) {\footnotesize{0.2}};
\draw[very thick](-0.1,4.29) -- (0.1,4.29);
\node at (-0.5,4.29) {\footnotesize{0.3}};

\filldraw[fill=blue!25]
(1,0.25) -- (1,1.91) -- (2,1.91) --  (2,0.25) -- (1,0.25);

\draw[very thick](1,0.75) -- (2,0.75);

\draw[dashed] (1.5,0.05)--(1.5,0.25);
\draw (1.2,0.05)--(1.8,0.05);

\draw[dashed] (1.5,1.91)--(1.5,4.04);
\draw (1.2,4.04)--(1.8,4.04);

\node at (1.5,-0.5) {\footnotesize{Cluster \#1}};

\filldraw[fill=blue!25]
(4,0.13) -- (4,2.32) -- (5,2.32) --  (5,0.13) -- (4,0.13);

\draw[very thick](4,1.39) -- (5,1.39);

\draw[dashed] (4.5,0.05)--(4.5,0.13);
\draw (4.2,0.05)--(4.8,0.05);

\draw[dashed] (4.5,2.32)--(4.5,4.60);
\draw (4.2,4.60)--(4.8,4.60);



\node at (4.5,-0.5) {\footnotesize{Cluster \#2}};

\filldraw[fill=blue!25]
(7,0.28) -- (7,1.93) -- (8,1.93) --  (8,0.28) -- (7,0.28);

\draw[very thick](7,1.25) -- (8,1.25);

\draw[dashed] (7.5,0.28)--(7.5,0.05);
\draw (7.2,0.05)--(7.8,0.05);

\draw[dashed] (7.5,1.93)--(7.5,3.85);
\draw (7.2,3.85)--(7.8,3.85);

\node at (7.5,-0.5) {\footnotesize{Cluster \#3}};

\filldraw[fill=blue!25]
(10,0.65) -- (10,1.68) -- (11,1.68) --  (11,0.65) -- (10,0.65);

\draw[very thick](10,0.93) -- (11,0.93);

\draw[dashed] (10.5,0.65)--(10.5,0.18);
\draw (10.2,0.18)--(10.8,0.18);

\draw[dashed] (10.5,1.68)--(10.5,2.22);
\draw (10.2,2.22)--(10.8,2.22);

\draw (10.45,3.80)--(10.55,3.90);
\draw (10.55,3.80)--(10.45,3.90);

\node at (10.5,-0.5) {\footnotesize{Cluster \#4}};

\end{tikzpicture}
\end{subfigure}

\caption{Distribution of inconsistency ratios $\mathit{CR}$, dataset S4, $k=4$ clusters}
\label{Fig3}
\end{figure}


The inconsistency ratios of the cluster centres are 0.007, 0.008, 0.028, 0.063 in the case of $D_1$, while 0.007, 0.009, 0.022, 0.013 in the case of $D_3$. Similar to the results in Section~\ref{Sec41}, the clusters are relatively uniform with respect to the level of inconsistency as can be seen in Figure~\ref{Fig3}. Even though the cluster centres are slightly less inconsistent for $D_3$, each centre has an acceptable inconsistency according to the famous 10\% rule of thumb. Naturally, this is not guaranteed in other datasets but the LP model of the $k$-medoids problem (Section~\ref{Sec3}) might contain any (linear) restriction on the inconsistency of the cluster centres.

As the cluster centres are PCMs, the associated priority weights can help the interpretation of the clusters. To that end, we have used the geometric mean method (Section~\ref{Sec21}).
For the dissimilarity measure $D_1$, they are as follows:
\[
\left[
\begin{array}{c}
0.495\\
0.291\\
0.067\\
0.148
\end{array}
\right], \qquad
\left[
\begin{array}{c}
0.155\\
0.114\\
0.310\\
0.420
\end{array}
\right], \qquad
\left[
\begin{array}{c}
0.260\\
0.102\\
0.049\\
0.589
\end{array}
\right], \qquad
\left[
\begin{array}{c}
0.511\\
0.054\\
0.180\\
0.254
\end{array}
\right].
\]
On the other hand, for index $D_3$:
\[
\left[
\begin{array}{c}
0.495\\
0.291\\
0.067\\
0.148
\end{array}
\right], \qquad
\left[
\begin{array}{c}
0.109\\
0.089\\
0.284\\
0.517
\end{array}
\right], \qquad
\left[
\begin{array}{c}
0.403\\
0.105\\
0.339\\
0.153
\end{array}
\right], \qquad
\left[
\begin{array}{c}
0.368\\
0.065\\
0.113\\
0.455
\end{array}
\right].
\]
Against our conjecture, the second and the third summer houses do not have the highest priority in any cluster; the first and the last alternatives are the best in two clusters, respectively.

Aggregating the individual pairwise comparison matrices by the geometric means of the entries \citep{AczelSaaty1983} and applying $\mathit{LLSM}$ to the common matrix results in the following priority vector:
\begin{equation} \label{Agg_weights}
\left[
\begin{array}{c}
0.410 \\
0.164 \\
0.146 \\
0.279
\end{array}
\right].   
\end{equation}
The implied ranking $1 \succ 4 \succ 2 \succ 3$ differs from the ranking in any cluster centre, thus, the aggregated ranking does not correspond to the preferences of any group.

Clustering may provide an aggregation procedure if there is only one cluster, whose centre represents ``best'' the whole set of PCMs. The cluster centres are
\[
\mathit{CC}_1 = \left[
\begin{array}{cccc}
1 &	2.000&	3.000&	1.500\\
0.500&	1 &	2.000&	0.500\\
0.333&	0.500&	1 &	0.250\\
0.667&	2.000&	4.000&	1 
\end{array}
\right] \quad \text{and} \quad
\mathit{CC}_3 = \left[
\begin{array}{cccc}
1 &	2.000&	1.500&	1.000\\
0.500&	1 &	0.667&	0.667\\
0.667&	1.500&	1 &	0.667\\
1.000&	1.500&	1.500&	1 
\end{array}
\right]
\]
for measures $D_1$ and $D_3$, respectively, and the corresponding weight vectors are
\[
\left[
\begin{array}{c}
0.381\\
0.185\\
0.099\\
0.334
\end{array}
\right] \qquad \text{and} \qquad
\left[
\begin{array}{c}
0.319\\
0.166\\
0.219\\
0.296
\end{array}
\right],
\]
which lead to the rankings $1 \succ 4 \succ 2 \succ 3$ and $1 \succ 4 \succ 3 \succ 2$, respectively. The rankings are identical to the ranking derived from the weights~\eqref{Agg_weights} associated with the aggregated matrix, except for a rank reversal at the bottom in the case of dissimilarity measure $D_3$. On the other hand, the relative weights of the first and the last alternatives are closer to each other according to both clusters than according to the result obtained by a reasonable aggregation of the individual matrices given in~\eqref{Agg_weights}.

\subsection{The appropriate number of clusters} \label{Sec43} 

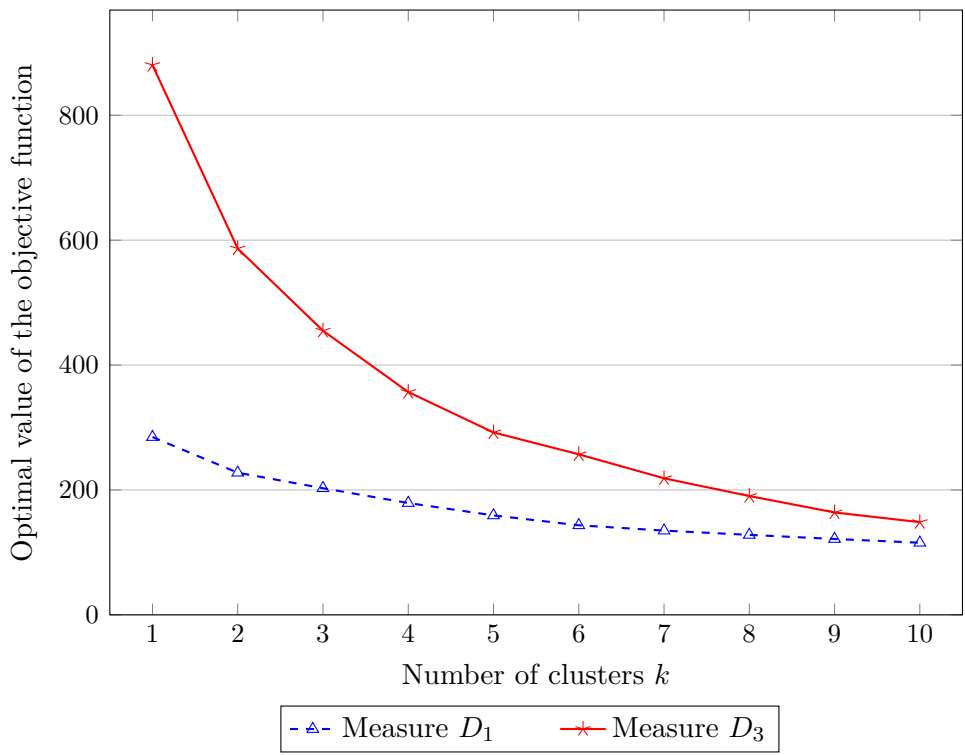
\begin{figure}[t!]
\centering

\begin{tikzpicture}
\begin{axis}[
xlabel = Number of clusters $k$,
x label style = {font=\small},
x tick label style={/pgf/number format/1000 sep=},
xtick = data,
ylabel = Optimal value of the objective function,
y label style = {font=\small},
width = 0.8\textwidth,
height = 0.6\textwidth,
ymajorgrids = true,
xmin = 0.5,
xmax = 10.5,
ymin = 0,
legend style = {font=\small,at={(0.2,-0.15)},anchor=north west,legend columns=2},
legend entries = {Measure $D_1\qquad$,Measure $D_3$}
] 
\addplot [blue, thick, dashed, mark=triangle, mark size=2.5pt, mark options={solid,thin}] coordinates {
(1,284.71651232473)
(2,227.609073926208)
(3,202.753133921767)
(4,179.064869588144)
(5,159.249709500747)
(6,143.414697066035)
(7,134.887464659981)
(8,128.070518847741)
(9,121.491541317771)
(10,115.482236811349)
};
\addplot [red, thick, solid, mark=star, mark size=3pt, mark options={solid,thin}] coordinates {
(1,880.402650073308)
(2,586.536975615243)
(3,455.011075803225)
(4,356.818341319861)
(5,292.087200249852)
(6,257.048311383301)
(7,218.686435412982)
(8,190.266886826488)
(9,163.970220139398)
(10,148.418236003698)
};
\end{axis}
\end{tikzpicture}

\caption{The optimum as a function of the number of clusters, dataset S4}
\label{Fig4}
\end{figure}


There is no unique or standard procedure to determine the appropriate number of clusters. Naturally, one might have knowledge about the number of clusters based on expert opinion or the demographic background of the DMs.

Nonetheless, it is a more interesting case if no previous information is available on the expected/desired number of clusters. A popular choice could be the so-called ``elbow'' method: the sum of distances from the cluster centres (the objective function of the LP in Section~\ref{Sec3}) is plotted for all relevant numbers of clusters, and the value for which the decreasing trend flattens is chosen. The corresponding chart for sample S4 is shown in Figure~\ref{Fig4}, suggesting that $k=4$ or $k=5$ clusters is the best option as a further increase in $k$ does not lead to a substantial reduction in the optimal value of the objective function.

A more sophisticated version of the ``elbow'' method is the gap statistic  \citep{TibshiraniWaltherHastie2001}. Although it ``\emph{is designed to be applicable to any clustering method and distance measure}'' \citep[p.~411]{TibshiraniWaltherHastie2001}, during the implementation uniformly distributed samples are generated in a multivariate Euclidean space. Therefore, this method does not seem to be suitable for our problem.

Several other approaches are available to obtain the optimal number of clusters but the majority of them are based on the sum of squares. Although the variance could be calculated, we think it has no common meaning for the dissimilarity measures considered.

\begin{table}[t!]
    \centering
\caption{Mean silhouette values, sample S4}
\label{Table3}
    \rowcolors{1}{}{gray!20}
    \begin{tabularx}{0.6\textwidth}{cCC} \toprule
Number of clusters & Measure $D_1$ & Measure $D_3$ \\ \bottomrule
2 & 0.286 & 0.467 \\
3 & 0.285 & 0.407 \\
4 & 0.260 & 0.503 \\
5 & 0.310 & 0.510 \\
6 & 0.293 & 0.494 \\
7 & 0.302 & 0.477 \\
8 & 0.321 & 0.488 \\
9 & 0.337 & 0.535 \\
10 & 0.345 & 0.568 \\
11 & 0.348 & 0.556 \\
12 & 0.325 & 0.571 \\
13 & 0.325 & 0.553 \\
14 & 0.330 & 0.545 \\
15 & 0.327 & 0.538 \\
16 & 0.324 & 0.514 \\
17 & 0.308 & 0.512 \\
18 & 0.310 & 0.523 \\
19 & 0.306 & 0.521 \\
20 & 0.310 & 0.518 \\ \toprule
    \end{tabularx}
\end{table}

One exception is the silhouette method \citep{KaufmanRousseeuw1990}, which uses the dissimilarity matrix $\Delta$ but does not rely on any specific property of Euclidean spaces. The average silhouette values are given in Table~\ref{Table3}. They tend to grow for small values of $k$, the maximum is reached at 11 (12) clusters in the case of index $D_1$ ($D_3$). However, these numbers seem to be unreasonably high.

\input{Figure5_dendrograms}

Another common method is using a hierarchical clustering method and choosing the number of clusters based on dendrograms, which are presented in Figure~\ref{Fig5}. Note that Figure~\ref{Fig5a} has been derived by the square of $D_1$ to make the structure more visible. The dendrogram shows the level of dissimilarity at which the corresponding objects are merged into the same cluster. 
According to these dendrograms, the number of clusters $k$ should be between three and six in our case.

To conclude, we currently could not recommend any method that immediately gives the number of clusters. The problem is somewhat analogous to the choice of the dissimilarity measure, where there is no perfect solution, too.

\subsection{Typical patterns in pairwise comparison matrices} \label{Sec44} 

This section considers another possible application of the proposed clustering approach. In particular, we are not interested in the preferences among the alternatives (i.e. which alternative is the best) but look for typical patterns in the structure of preferences and for common sources of inconsistency. Therefore, the alternatives are relabelled: the most preferred alternative will be the first, the second most preferred the second, and so on. To that end, the alternatives are ranked by the geometric mean method (Section~\ref{Sec21}).
For dataset S4, dissimilarity measure $D_1$, and four clusters, which is the case presented in Section~\ref{Sec42}, the cluster centres are:
\[
\mathit{CC}_1^{(1)} = \left[
\begin{array}{cccc}
1 &	2.000&	5.000&	5.000\\
0.500&	1 &	3.000&	3.000\\
0.200&	0.333&	1 &	1.500\\
0.200&	0.333&	0.667&	1 
\end{array}
\right], \qquad
\mathit{CC}_1^{(2)} = \left[
\begin{array}{cccc}
1 &	2.000&	2.500&	4.000\\
0.500&	1 &	1.500&	3.500\\
0.400&	1.667&	1 &	3.000\\
0.250&	0.286&	0.333&	1 
\end{array}
\right],
\]
\[
\mathit{CC}_1^{(3)} = \left[
\begin{array}{cccc}
1 &	{\it 5.000}&	9.000&	{\bf 9.000}\\
0.200&	1 &	{\it 7.000}&	7.000\\
0.111&	0.143&	1 &{\it	5.000}\\
0.111&	0.143&	0.200&	1 
\end{array}
\right], \qquad
\mathit{CC}_1^{(4)} = \left[
\begin{array}{cccc}
1 &	3.000&	5.000&	7.000\\
0.333&	1 &	3.000&	6.000\\
0.200&	0.333&	1 &	4.000\\
0.143&	0.167&	0.250&	1 
\end{array}
\right],
\]
and the clusters contain 14, 18, 10, and 26 matrices, respectively.

\begin{figure}[t!]
\centering

\begin{tikzpicture}

\draw[very thick](-0.5,0) -- (12.5,0);

\draw[->, very thick](0,-0.5) -- (0,5.5);

\draw[very thick](-0.1,1.43) -- (0.1,1.43);
\node at (-0.4,1.43) {0.1};
\draw[very thick](-0.1,2.86) -- (0.1,2.86);
\node at (-0.4,2.86) {0.2};
\draw[very thick](-0.1,4.29) -- (0.1,4.29);
\node at (-0.4,4.29) {0.3};

\filldraw[fill=blue!25]
(1,0.14) -- (1,1.00) -- (2,1.00) --  (2,0.14) -- (1,0.14);

\draw[very thick](1,0.35) -- (2,0.35);

\draw[dashed] (1.5,0.05)--(1.5,0.14);
\draw (1.2,0.05)--(1.8,0.05);

\draw[dashed] (1.5,1.00)--(1.5,1.47);
\draw (1.2,1.47)--(1.8,1.47);

\draw (1.45,4.45)--(1.55,4.55);
\draw (1.55,4.45)--(1.45,4.55);

\draw (1.5,0)--(1.5,-0.2);
\node at (1.5,-0.5) {\scriptsize{Cluster \#1}};

\filldraw[fill=blue!25]
(4,0.23) -- (4,0.73) -- (5,0.73) --  (5,0.23) -- (4,0.23);

\draw[very thick](4,0.26) -- (5,0.26);

\draw[dashed] (4.5,0.05)--(4.5,0.23);
\draw (4.2,0.05)--(4.8,0.05);

\draw[dashed] (4.5,0.73)--(4.5,1.22);
\draw (4.2,1.22)--(4.8,1.22);

\draw (4.45,2.27)--(4.55,2.37);
\draw (4.55,2.27)--(4.45,2.37);

\draw (4.5,0)--(4.5,-0.2);
\node at (4.5,-0.5) {\scriptsize{Cluster \#2}};

\filldraw[fill=blue!25]
(7,2.35) -- (7,3.85) -- (8,3.85) --  (8,2.35) -- (7,2.35);

\draw[very thick](7,3.00) -- (8,3.00);

\draw[dashed] (7.5,2.12)--(7.5,2.35);
\draw (7.2,2.12)--(7.8,2.12);

\draw[dashed] (7.5,3.85)--(7.5,4.14);
\draw (7.2,4.14)--(7.8,4.14);

\draw (7.5,0)--(7.5,-0.2);
\node at (7.5,-0.5) {\scriptsize{Cluster \#3}};

\filldraw[fill=blue!25]
(10,0.89) -- (10,1.70) -- (11,1.70) --  (11,0.89) -- (10,0.89);

\draw[very thick](10,1.21) -- (11,1.21);

\draw[dashed] (10.5,0.23)--(10.5,0.89);
\draw (10.2,0.23)--(10.8,0.23);

\draw[dashed] (10.5,1.70)--(10.5,2.35);
\draw (10.2,2.35)--(10.8,2.35);

\draw (10.45,4.32)--(10.55,4.42);
\draw (10.55,4.32)--(10.45,4.42);

\draw (10.45,2.99)--(10.55,3.09);
\draw (10.55,2.99)--(10.45,3.09);

\draw (10.5,0)--(10.5,-0.2);
\node at (10.5,-0.5) {\scriptsize{Cluster \#4}};
\end{tikzpicture}

\caption{Distribution of inconsistency ratios $\mathit{CR}$, dataset S4, $k=4$ clusters, measure $D_1$, relabelled pairwise comparison matrices}
\label{Fig6}

\end{figure}
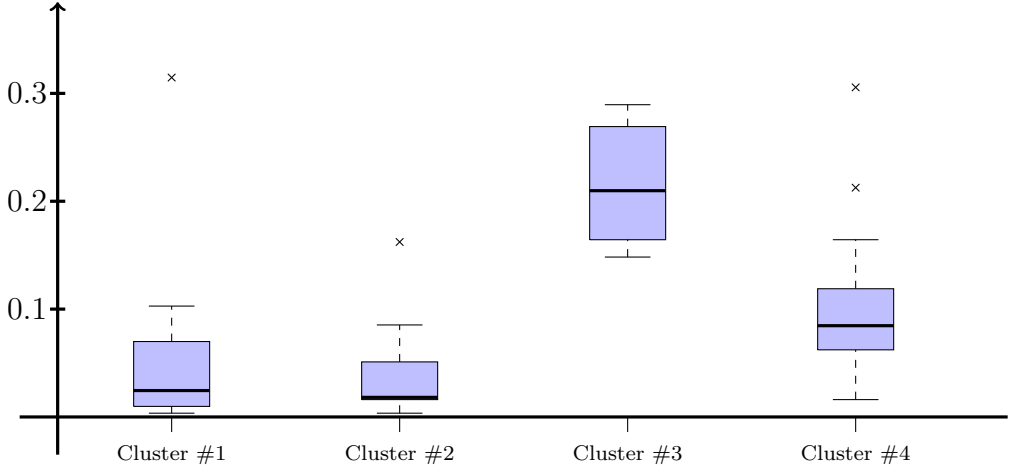

Figure \ref{Fig6} reports the boxplots of inconsistency ratios in all clusters analogously to Figure~\ref{Fig3a}. Contrary to the previous results, now there exists a cluster (Cluster \#3) that contains highly inconsistent matrices. We have tried to find a pattern in these 10 matrices. Consistency implies $a_{14} = a_{12} \times a_{23} \times a_{34}$. This is definitely not true for $\mathit{CC}_1^{(3)}$, where $a_{41}=9$ but $a_{12} \times a_{23} \times a_{34} = 175$.
The other matrices of Cluster \#3 are similar to the cluster centre, including the property that all the six comparisons above the diagonal are higher than one.
In other words, while these DMs have clear ordinal preferences, they are not able to appropriately translate them into numbers.

\section{Discussion} \label{Sec5}

In the following, two issues are considered: Section~\ref{Sec51} provides some thoughts on the potential contribution of clustering to the analysis of group decision making problems based on pairwise comparisons, while Section~\ref{Sec52} discusses how demographic variables can be utilised in the proposed clustering approach.

\subsection{The potential role of cluster analysis in the study of pairwise comparison matrices} \label{Sec51}

Pairwise comparison matrices can be classified in several ways such as by the best alternative, by the ranking of the alternatives, or by their level of inconsistency. However, this procedure usually requires an initial assumption on which the classification is carried out, and the groups are created after a reduction of dimension, which certainly loses some information about the original pairwise comparisons. For instance, inconsistency ``measures'' the matrices on a one-dimensional scale, hence, many ``qualitatively'' different matrices might have the same inconsistency. Although this is not necessarily a problem if inconsistency is the focus of the analysis, it can be misleading if dissimilar matrices are grouped together.

On the other hand, cluster analysis is a method of exploratory data analysis that has no underlying idea of the classification, and considers the individual pairwise comparison matrices without any transformation. Naturally, this flexibility makes the result dependent on the measure of similarity (Section~\ref{Sec22}) used. Nonetheless, clustering is able to uncover surprising patterns in a dataset that are difficult to recognise without a particular investigation as Section~\ref{Sec4} illustrates:
\begin{itemize}
\item
A typo (mistakenly reversed comparison) is found in a pairwise comparison matrix (Section~\ref{Sec41});
\item
Pairwise comparison matrices can be quite different---that is, belong to different clusters---even if their inconsistencies are close and the best alternatives coincide (Section~\ref{Sec42});
\item
A potential source of inconsistency (inappropriate conversion of robust ordinal preferences into cardinal values) is identified (Section~\ref{Sec44}).
\end{itemize}
Hopefully, our paper will inspire further instructive applications of clustering pairwise comparison matrices.

\subsection{Incorporating demographic factors into the cluster analysis} \label{Sec52} 

\begin{table}[t!] 
\centering
\caption{An illustrative contingency table for clusters and gender}
\label{Table4}
\begin{threeparttable}
    \rowcolors{1}{}{gray!20}
    \begin{tabularx}{0.5\textwidth}{LCCC} \toprule
        Cluster & Male & Female & $\sum$ \\ \bottomrule
        $\mathit{CL}_1^{(1)}$ & 7 & 19 & 26 \\
        $\mathit{CL}_1^{(2)}$ & 5 & 10 & 15 \\
        $\mathit{CL}_1^{(3)}$ & 1 & 8 & 9 \\
        $\mathit{CL}_1^{(4)}$ & 3 & 15 & 18 \\ \bottomrule
        $\sum$ & 16 & 52 & 68 \\ \toprule
    \end{tabularx}
\begin{tablenotes} \footnotesize
\item
Dataset S4, four clusters, measure $D_1$
\end{tablenotes}
\end{threeparttable}
\end{table}

In many application, further information is available about the DMs, and the clusters might be determined at least partially by demographic variables. The suggested clustering method can help this kind of analysis as well. The dataset studied in Section~\ref{Sec4} contains the age category and the gender of the DMs, furthermore, they faced three different filling orders.
Table~\ref{Table4} reveals the joint distribution of the four clusters from Section~\ref{Sec42} and gender. Although the proportion of male and female students is not the same in all clusters, the differences are statistically insignificant. We have found no significant association with the clusters for any other variables.

Naturally, in other applications more information may be available from the DMs, and the cluster analysis can be based on demographic factors. Another possibility could be to define the distance between decision makers by considering both their preferences and other, such as demographic, characteristics; or a bicluster approach where both preferences and demographic values should be similar in any cluster.

\section{Conclusions} \label{Sec6}

This paper has provided a new perspective on group decision making, especially large-scale group decision making, by proposing the $k$-medoids clustering for pairwise comparison matrices.
Our method has some advantages over other solutions:
\begin{itemize}
\item 
it is independent of the specification of the weighting method;
\item 
it can handle incomplete data such that the impact of incomplete pairwise comparison matrices, which contain less information, is inherently reduced;
\item 
it is more robust to outliers than the $k$-means clustering algorithm;
\item
the cluster centres are guaranteed to be individual pairwise comparison matrices, making them easier to accept by the decision-makers;
\item
its LP formulation allows for adding various restrictions, for example, regarding the inconsistency of the cluster centres.
\end{itemize}
The suggested approach has been used to analyse the experimental dataset of \citet{BozokiDezsoPoeszTemesi2013}

Hopefully, all practitioners dealing with data from large-scale group decision-making may benefit from using the clustering model presented here. Further investigations are especially welcome because, at the moment, there are few results on choosing the dissimilarity measure underlying the $k$-medoids algorithm or the appropriate number of clusters $k$.

\section*{Acknowledgements}
\addcontentsline{toc}{section}{Acknowledgements}
\noindent
Four anonymous reviewers gave useful remarks on earlier drafts. \\
The research was supported by the National Research, Development and Innovation Office under Grants FK 145838 and TKP2021-NKTA-01 NRDIO, and the J\'anos Bolyai Research Scholarship of the Hungarian Academy of Sciences.

\bibliographystyle{apalike}
\bibliography{All_references}

\end{document}